\documentclass[11pt]{article}

\usepackage{amsmath}
\usepackage{amssymb}
\usepackage{amsthm}
\usepackage{multirow}
\usepackage{color}
\usepackage{longtable}
\usepackage{array}
\usepackage{url}
\usepackage{tikz}

\newtheorem{theorem}{Theorem}[section]

\newtheorem{lemma}[theorem]{Lemma}
\newtheorem{example}[theorem]{Example}

\newtheorem{corollary}[theorem]{Corollary}

\newtheorem{remark}{Remark}[section]

\title{A Generalization of Combinatorial Designs Related to Almost Difference Sets}

\author{
Jerod Michel\thanks{Corresponding author. Email Address: contextolibre@gmail.com.}
\thanks{J. Michel is with the Department of Mathematics, Zhejiang University, Hangzhou 310027, China.}
, Baokun Ding\thanks{B. Ding is with the Department of Mathematics, Zhejiang University, Hangzhou 310027, China.} \\
}

\begin{document}

\date{}\maketitle

\begin{abstract}
In this paper we study a certain generalization of combinatorial designs related to almost difference sets, namely the $t$-adesign, which was coined by Cunsheng Ding in 2015. It is clear that $2$-adesigns are a kind of partially balanced incomplete block design which naturally arise in many combinatorial and statistical problems. We discuss some of their basic properties and give several constructions of $2$-adesigns (some of which correspond to new almost difference sets, and others of which correspond to new almost difference families), as well as two constructions of $3$-adesigns.  We also discuss some basic properties of their incidence matrices and codes.

\smallskip

\end{abstract}

%\begin{keywords}
%Difference sets, almost difference sets, cyclotomic cosets, combinatorial designs.
%\end{keywords}

\section{Introduction}\label{sec1}
Combinatorial designs have extensive applications in many fields, including finite geometry \cite{DEM}, \cite{HIR}, design of experiments \cite{FISH}, \cite{BOS}, cryptography \cite{CDR}, \cite{STIN}, and authentication codes and secret sharing schemes \cite{OKS}, \cite{STIN}.
We will assume some familiarity with combinatorial design theory. A $t$-$(v,k,\lambda)$ {\it design} (with $v>k>t>0$) is an incidence structure $(V,\mathcal{B})$ where $V$ is a set of $v$ points and $\mathcal{B}$ is a collection of $k$-subsets of $V$ (called blocks), such that any $t$-subset of $V$ is contained in exactly $\lambda$ blocks \cite{STIN}. When $t=2$, a $t$-design is sometimes referred to as a balanced incomplete block design. Denoting the number of blocks by $b$ and the number of blocks containing a given point by $r$, the identities \[
bk=vr \] and \[ r(k-1)=(v-1)\lambda \] restrict the possible parameter sets. A $t$-$(v,k,\lambda)$ design in which $b=v$ and $r=k$ is called {\it symmetric}, and any two blocks meet in $\lambda$ points. A $t$-$(v,k,\lambda)$ design is called {\it quasi-symmetric} if two blocks meet in $s_{1}$ points or in $s_{2}$ points (where $s_{1}<s_{2}$). The {\it dual} $(V,\mathcal{B})^{\perp}$ of an incidence structure $(V,\mathcal{B})$ is the incidence structure $(\mathcal{B},V)$ with the roles of points and blocks interchanged. A symmetric incidence structure is always isomorphic to its dual. The Bruck-Ryser-Chowla theorem provides necessary conditions for the existence of symmetric $2$-$(v,k,\lambda)$ designs.

\begin{theorem} (\cite{STIN}, \cite{LAN}) Suppose there exists a symmetric $2$-$(v,k,\lambda)$ design. Set $n=k-\lambda$. Then \begin{enumerate} \item if $v$ is even then $n$ is a square,
\item if $v$ is odd then the equation \[ z^{2}=nx^{2}+(-1)^{(v-1)/2}\lambda y^{2}\] has a solution in integers $x,y,z$ not all zero. \end{enumerate}
\end{theorem}
\par
Difference sets \cite{HALL} and almost difference sets \cite{NOW} also have applications in many areas such as digital communications \cite{YZ}, \cite{DHM}, sequence design \cite{TANG}, \cite{NAM},  and CDMA and cryptography \cite{CDR}. We will also assume familiarity  with some facts about difference sets and almost difference sets. Let $G$ be a finite additive group with identity 0. Let $k$ and $\lambda$ be positive integers such that $2\leq k < v$. A $(v,k,\lambda)$ {\it difference set} in $G$ is a subset $D \subseteq G$ that satisfies the following properties: \begin{enumerate} \item $|D|=k$,
\item the multiset $\{x-y \mid x,y \in D, x \neq y \}$ contains every member of $G-\{0\}$ exactly $\lambda$ times.\end{enumerate} Almost difference sets are a generalization of difference sets. A $(v,k,\lambda,t)$ {\it almost difference set} in $G$ is a subset $D \subseteq G$ that satisfies the following properties: \begin{enumerate} \item $|D|=k$,
    \item the multiset $\{x-y \mid x,y \in D, x \neq y\}$ contains $t$ members of $G-\{0\}$ which appear $\lambda$ times and $v-1-t$ members of $G-\{0\}$ which appear $\lambda+1$ times. \end{enumerate}
Due to their having extensive applications, it is worthwhile to study the combinatorial objects arising from almost difference sets and almost difference families. In Section \ref{sec2} we introduce the generalizations and discuss some basic properties. In Section \ref{sec3} we give three constructions of $2$-adesigns from quadratic residues, in Section \ref{sec4} we give several constructions of $2$-adesigns which are almost difference families, in Section \ref{sec4.5} we give several constructions of $2$-adesigns from symmetric $t$-designs, in Section \ref{sec5} we give two constructions of $3$-adesigns, and in Section \ref{sec5.5} we discuss the codes of $t$-adesigns and related structures. Section \ref{sec6} concludes the paper.

\section{Preliminaries}\label{sec2}

Let $G$ be an additive group of order $v$. A $k$-element subset $D \subseteq G$ has {\it difference levels} $\mu_{1} < \cdots < \mu_{s}$ if there exist integers $t_{1},...,t_{s}$ such that the multiset \[ M=\{g-h \mid g,h \in D\} \] contains exactly $t_{i}$ members of $G-\{0\}$ each with multiplicity $\mu_{i}$ for all $i$, $1 \leq i \leq s$. We will denote the $t_{i}$ members of the multiset $M$ with multiplicity $\mu_{i}$ by $T_{i}$. Note that the $T_{i}$'s form a partition of $G-\{0\}$. It is easy to see that in the case where $s=1$, $D$ is a difference set \cite{HALL}, and in the case where $s=2$ and $\mu_{2}=\mu_{1}+1$, $D$ is an almost difference set \cite{NOW}. For the remainder of this correspondence all groups are assumed to be additive. The basic equation describing a $k$-element subset $D \subseteq G$ with difference levels $\mu_{1}<\cdots<\mu_{s}$ is given by \begin{equation}\label{eq1}
\mu_{1}t_{1}+\cdots+\mu_{s-1}t_{s-1}+\mu_{s}(v-1-\sum_{i=1}^{s-1}t_{i})=k(k-1) \end{equation}
\par
Let $V$ be a $v$-set and $\mathcal{B}$ a collection of subsets of $V$, called blocks, each having cardinality $k$. If there are positive integers $\mu_{1} < \cdots < \mu_{s}$ such that every subset of $V$ of cardinality $t$ is incident with exactly $\mu_{i}$ blocks for some $i$, $1 \leq i \leq s$, and for each $i$, $1 \leq i \leq s$ there exists a subset of $V$ of cardinality $t$ that is incident with exactly $\mu_{i}$ blocks, then we say that the incidence structure $(V,\mathcal{B})$ has {\it t-levels} $\mu_{1} < \cdots < \mu_{s}$. We denote $|\mathcal{B}|$ by $b$. An incidence structure $(V,\mathcal{B})$ is called {\it symmetric} if $b=v$. In the case where $s=2$, $(V,\mathcal{B})$ is a partially balanced incomplete block design, and if $\mu_{2}=\mu_{1}+1$, we call $(V,\mathcal{B})$ a $t$-$(v,k,\mu_{1})$ {\it adesign} (or simply a $t$-{\it adesign}), which was coined by Ding in \cite{CUN}. It is easy to see that in the case where $s=1$, $(V,\mathcal{B})$ is simply a $t$-design \cite{STIN}.

We call the set $\{D+g \mid g \in G\}$ of translates of $D$, denoted by Dev$D$, the {\it development} of $D$.

\begin{theorem}\label{th1} Let $G$ be an additive group of order $v$. Let $D \subseteq G$ be a $k$-subset with difference levels $\mu_{1} < \cdots < \mu_{s}$. Then the incidence structure $(G,$ Dev$D)$ is symmetric with $2$-levels $\mu_{1} < \cdots < \mu_{s}$.
\end{theorem}
\begin{proof} Assume the conditions for $D \subseteq G$. Suppose that $x,y \in G$ and $x \neq y$. We first show that the number of members $g \in G$ such that $\{x,y\} \subseteq D+g$ is $\mu_{j}$ for some $j$, $1 \leq j \leq s$. Let $d=x-y$. Since $D$ is a $(v,k,s)$ multi-difference set with levels $\mu_{1} < \cdots < \mu_{s}$, the number of ordered pairs $(x',y')$ such that $x',y' \in D$ and $x'-y'=d$ can be found among the $\mu_{i}$'s. Denote these by $(x_{i}^{j},y_{i}^{j})$, $1 \leq i \leq \mu_{j}$, $1 \leq j \leq s$, if $x_{i}^{j}-y_{i}^{j}=d$ is among the $t_{j}$ members of $G-\{0\}$, which we may denote by $T_{j}$, that appear as a difference of members of $D$ exactly $\mu_{j}$ times. Now suppose $d \in T_{j'}$ for fixed $j'$, $1 \leq j' \leq s$. Then for $1 \leq i \leq \mu_{j}$, let $g_{i}=-x_{i}^{j}+x$. Then $g_{i}^{j'}=-y_{i}^{j'}+y$ and $\{x,y\}=\{x_{i}^{j'}+g_{i},y_{i}^{j'}+g_{i}\} \subseteq D+g_{i}$. Note that the $g_{i}$'s are distinct because the $x_{i}^{j'}$'s are distinct. Thus, there are at least $\mu_{j'}$ members $g \in G$ such that $\{x,y\} \subseteq D+g$.
\par
Conversely, suppose that there are exactly $l$ members $h \in G$ such that $\{x,y\} \subseteq D+h$, namely $h_{1},...,h_{l}$. Notice that $(x-h_{i})+(h_{i}-y)=x-y=d$ and $\{x-h_{i},y-h_{i}\} \subseteq D$ for each $i$, $1 \leq i \leq l$. The $h_{i}$'s are distinct, so there are $l$ ordered pairs $(x',y')$ such that $x',y' \in D$ and $x'-y'=d$. The $T_{i}$'s form a partition of $G-\{0\}$ and so $d$ must be a member of one of them. Suppose $d \in T_{j'}$. Then $l \leq \mu_{j'}$ and this gives the result.
\end{proof}

We also have the following corollary whose proof is a trivial consequence of Theorem \ref{th1}.

\begin{corollary} Let $D$ be a $(v,k,\lambda)$ almost difference set in an Abelian group $G$. Then $(G,$ Dev$D)$ is a $2$-$(v,k,\lambda)$ adesign.
\end{corollary}

Let $(V,\mathcal{B})$ be an incidence structure with $t$-levels $\mu_{1} < \cdots < \mu_{s}$. Let $A$ be a $v$ by $b$ matrix whose rows and columns are indexed by points and blocks respectively and whose $(i,j)$-th entry is $1$ if the point corresponding to the $i$th row is incident with the block corresponding to the $j$th row, and $0$ otherwise. We call $A$ the {\it incidence matrix} of $(V,\mathcal{B})$. We will denote the $n\times n$ identity and all-one matrices by $I_{n}$ and $J_{n}$ respectively, or, when it is clear from the context, simply by $I$ and $J$.

\begin{lemma}\label{le1} Let $D$ be a $k$-subset of an Abelian group $G$ of cardinality $v$ with difference levels $\mu_{1} < \cdots < \mu_{s}$. Let $A$ be the $v \times v$ incidence matrix of the symmetric incidence structure $(G,$ Dev$D)$. Then \begin{equation}\label{eq11} A^{T}A=AA^{T}=kI+\mu_{1}A_{1}+\cdots +\mu_{s-1}A_{s-1}+\mu_{s}(J-\sum_{i=1}^{s-1}A_{i}-I), \end{equation} where $A_{i}$ is a binary matrix whose $(g,h)$th entry is $1$ if the points of $G$ corresponding to $g$ and $h$ appear together in exactly $\mu_{i}$ blocks of $DevD$, and is $0$ otherwise.
\end{lemma}
\begin{proof} The second equality in (\ref{eq11}) is clear. We show that if $x,y \in G$ and $x \neq y$ then $|(D+x) \cap (D+y)|=\mu_{e}$ for some $e$, $1 \leq e \leq s$ if and only if $\{x,y\}$ is contained in exactly $\mu_{e}$ blocks.
\par
First we may denote the members of $D$ by $g_{1},...,g_{k}$. Suppose, without loss of generality, that $x=0$, so that we have \[ (D+x)\cap(D+y)=D\cap(D+y)=\{g_{\alpha_{1}},...,g_{\alpha_{\mu_{s}}}\} \]
where $1 \leq \alpha_{i} \leq k$ for each $i$, $1 \leq i \leq \mu_{e}$.
Let $D+h_{1},...,D+h_{k}$ be the $k$ translates of $D$ containing $0$ and, if necessary, relabel so that \[
g_{1}+h_{1}=\cdots = g_{k}+h_{k}=0. \] It suffices to show that $y$ is contained in exactly $\mu_{e}$ of the $k$ translates $D+h_{j}$, $1 \leq j \leq k$. Now fix $i'$, $1 \leq i' \leq \mu_{e}$. We have \[
g_{\alpha_{i'}} \in \{g_{1}+y,...,g_{k}+y\}\cap\{g_{1},...,g_{k}\}. \] There exists an integer $y'$ such that $g_{\alpha_{i'}}=g_{\alpha_{i'}-y'}+y$. Now we have $g_{\alpha_{i'}}+h_{\alpha_{i'}-y'}=g_{\alpha_{i'}-y'}+y+h_{\alpha_{i'}-y'}=y$. Since the $h_{j}$'s are distinct and there are $\mu_{e}$ of them, we have shown that the condition is necessary.
\par
Now suppose that $\{0,y\}$ is contained in exactly $\mu_{j}$ translates, namely $D+\kappa_{1},...,D+\kappa_{\mu_{e}}$. Relabel if necessary, so that $g_{1}+\kappa_{1}=\cdots=g_{\mu_{e}}=0$ where the $g_{i}$'s are distinct and contained in $D$, and $h_{1}+\kappa_{1}=\cdots=h_{\mu_{e}}=y$ where the $h_{i}$'s and are distinct and contained in $D$. Now notice that $y=h_{i}-g_{i}=h_{j}-g_{j}$ for all $i \neq j$, $1 \leq i,j \leq \mu_{e}$. This gives us $h_{i}+g_{j}=h_{j}+g_{i}$ so that $h_{i}+g_{j}+\kappa_{j}=g_{i}+h_{j}+\kappa_{j}$. This expression is clearly a member of $D\cap(D+y)$. Since the $h_{i}$'s and $g_{i}$'s are distinct and there are exactly $\mu_{e}$ of each, we have that the condition is sufficient.
\end{proof}

\begin{remark} We say that a $k$-subset $D$ in an Abelian group $G$ of cardinality $v$ is {\it cyclic} if $G$ is cyclic, and {\it Paley type} if $D=D^{-1}:=\{-x \mid x \in D\}$. Assume $G$ is cyclic and $D \subseteq G$ is Paley type. Let $A=(a_{i,j})$ be the incidence matrix of $(G,$ Dev$D)$. Notice that if $a_{i,j}=1$ then $a_{0,j-i}=1$ so that $j-i \in D$. But $i-j$ is also in $D$ since $D$ is Paley type. Thus $a_{0,i-j}=1$ whence $a_{j,i}=1$. It is clear that the condition is in fact necessary and sufficient. Thus, if $D$ has $s$ difference levels in a cyclic group $G$, and $A$ is the $v \times v$ incidence matrix of the symmetric incidence structure $(G,$ Dev$D)$, then $D$ is Paley type if and only if $A=A^{T}$.
\end{remark}

Next, we give some constructions of $2$-adesigns from almost difference sets.

\section{Constructions of $2$-adesigns from Quadratic Residues}\label{sec3}

Cyclotomic classes have proven to be a powerful tool for constructing difference sets and almost difference sets, e.g. see \cite{DHM}, \cite{NOW}, \cite{DING}. Let $q$ be a prime power, $\mathbb{F}_{q}$ a finite field, and $e$ a divisor of $q-1$. For a primitive element $\alpha$ of $\mathbb{F}_{q}$ let $D_{0}^{e}$ denote $\langle \alpha^{e} \rangle$, the multiplicative group generated by $\alpha^{e}$, and let \[ D_{i}^{e} = \alpha^{i}D_{0}^{e}, \text{ for } i=1,2,...,e-1. \] We call $D_{i}^{e}$ the {\it cyclotomic classes} of order $e$. The {\it cyclotomic numbers} of order $e$ are defined to be \[ (i,j)_{e} = \left| D_{i}^{e} \cap (D_{j}^{e} + 1) \right|. \]

It is easy to see there are at most $e^{2}$ different cyclotomic numbers of order $e$. When it is clear from the context, we simply denote $(i,j)_{e}$ by $(i,j)$.
The cyclotomic numbers $(h,k)$ of order $e$ have the following properties (\cite{D}):

\begin{eqnarray}\label{eq7}
(h,k) & = & (e-h,k-h), \\
(h,k) & = & \begin{cases}
(k,h),                         & \text{if } f \text{ even},\\
(k+\frac{e}{2},h+\frac{e}{2}), & \text{if } f \text{ odd}.
\end{cases}
\end{eqnarray}

Our first three constructions make use of quadratic residues. We will need the following lemma \cite{D}.

\begin{lemma}\label{le3} If $q \equiv 1$ (mod 4) then the cyclotomic numbers of order two are given by
\begin{eqnarray*}
(0,0) & = & \frac{q-5}{4},                         \\
(0,1) & = & (1,0) = (1,1) = \frac{q-1}{4}.
\end{eqnarray*} If $q \equiv 3$ (mod 4) then are given by
\begin{eqnarray*}
(0,1) & = & \frac{q+1}{4},                         \\
(0,1) & = & (1,0) = (1,1) = \frac{q-3}{4}.
\end{eqnarray*}
\end{lemma}

We are ready to give our first construction.

\begin{theorem}\label{th557} Let $q$ be an odd prime power and $\alpha$ a primitive member of $\mathbb{F}_{q}$. Define $C_{i}=\{z\in \mathbb{Z}_{q-1} \mid \alpha^{z}\in D_{i}^{2}-1\}$ for $i=1,2$. Then the incidence structure $(\mathbb{Z}_{q-1}\cup\{\infty\},Dev^{\infty}C_{0} \cup DevC_{1})$, where $Dev^{\infty}C_{0}$ denotes the blocks of $DevC_{0}$ each modified by adjoining the point ``$\infty$'', is a $2$-$(q,\frac{q-1}{2},\frac{q-5}{2})$ adesign.
\end{theorem}

\begin{proof} We will denote $\{\alpha^z\mid z\in C_{i}\}$ by $\alpha^{C_{i}}$. For $w \in \mathbb{Z}_{q-1}$ we have \[\left| C_{0} \cap (C_{0}+w) \right|=\left| \alpha^{C_{0}} \cap \alpha^{C_{0}+w} \right|\] which is \[\begin{cases}
\left|(D_{0}^{2}-\{1\})\cap(D_{0}^{2}-\{\alpha^{w}\}+(1-\alpha^{w}))\right| & \text{ if } w \text{ even},\\
\left|(D_{0}^{2}-\{1\})\cap(D_{1}^{2}-\{\alpha^{w}\}+(1-\alpha^{w}))\right| & \text{ if } w \text{ odd}.\end{cases}\] Since $\alpha^{w}(1-\alpha^{w})^{-1}=(1-\alpha^{w})^{-1}-1$, this becomes \[\begin{cases}
\left|(D_{0}^{2}-\{(1-\alpha^{w})^{-1}\})\cap(D_{0}^{2}-\{(1-\alpha^{w})^{-1}-1\}+1)\right| & \text{ if } w \text{ even},\\
\left|(D_{0}^{2}-\{(1-\alpha^{w})^{-1}\})\cap(D_{1}^{2}-\{(1-\alpha^{w})^{-1}-1\}+1)\right| & \text{ if } w \text{ odd},\end{cases}\] which simplifies to
\[\begin{cases}
\left|D_{0}^{2}\cap(D_{0}^{2}+1)-\{(1-\alpha^{w})^{-1}\}\right| & \text{ if } w \text{ even},\\
\left|D_{0}^{2}\cap(D_{1}^{2}+1)-\{(1-\alpha^{w})^{-1}\}\right| & \text{ if } w \text{ odd}.\end{cases}\]

There are four cases depending on the parity of $w$ and whether $(1-\alpha^{w})^{-1} \in D_{0}^{2}$ or $D_{1}^{2}$. By Theorem \ref{le3} we have \[\left| C_{0} \cap (C_{0}+w) \right|=\begin{cases}
(0,0)-1 & \text{ if } w \text{ even and } (1-\alpha^{w})^{-1}\in D_{0}^{2},\\
(0,0)   & \text{ if } w \text{ even and } (1-\alpha^{w})^{-1}\in D_{1}^{2},\\
(0,1)-1 & \text{ if } w \text{ odd and } (1-\alpha^{w})^{-1}\in D_{0}^{2},\\
(1,0)-1 & \text{ if } w \text{ even and } (1-\alpha^{w})^{-1}\in D_{0}^{2}.\end{cases}\]

Thus if $q \equiv 1($mod $4)$ then \[\left| C_{0} \cap (C_{0}+w) \right|=\begin{cases}
\frac{q-9}{4} & \text{ if } w \text{ even and } (1-\alpha^{w})^{-1}\in D_{0}^{2}, \\
\frac{q-5}{4} & \text{ otherwise}, \end{cases}\] and if $q \equiv 3 ($mod $4)$ then
\[\left| C_{0} \cap (C_{0}+w) \right|=\begin{cases}
\frac{q-3}{4} & \text{ if } w \text{ odd and } (1-\alpha^{w})^{-1}\in D_{0}^{2}
\text{ or if } w \text{ even and } (1-\alpha^{w})^{-1}\in D_{1}^{2}, \\
\frac{q-7}{4} & \text{ otherwise}. \end{cases}\]

Also, we have
\[\left| C_{1} \cap (C_{1}+w) \right|=\begin{cases}
\left|D_{1}^{2}\cap(D_{1}^{2}+1)+(1-\alpha^{w})\right| & \text{ if } w \text{ even},\\
\left|D_{1}^{2}\cap(D_{1}^{2}+1)+(1-\alpha^{w})\right| & \text{ if } w \text{ odd}.\end{cases}\]
Thus if $q \equiv 1 ($mod $4)$ then \[\left| C_{1} \cap (C_{1}+w) \right|=\begin{cases}
\frac{q-5}{4} & \text{ if } w \text{ even and } (1-\alpha^{w})^{-1}\in D_{1}^{2},\\
\frac{q-1}{4} & \text{ otherwise}. \end{cases}\] and if $q \equiv 3 ($mod $4)$ then
\[\left| C_{1} \cap (C_{1}+w) \right|=\begin{cases}
\frac{q+1}{4} & \text{ if } w \text{ odd and } (1-\alpha^{w})^{-1}\in D_{1}^{2},\\
\frac{q-3}{4} & \text{ otherwise}. \end{cases}\]

We need to compute the number of blocks of $(\mathbb{Z}_{q-1},DevC_{0}\cup DevC_{1})$ in which an arbitrary pair of points appear. Consider the incidence structures $(\mathbb{Z}_{q-1},DevC_{i})$ for $i=0,1$. Let $C_{i}^{\perp},(C_{i}+w)^{\perp}$ denote the points of the dual structures $(DevC_{i},\mathbb{Z}_{q-1})$ corresponding to the blocks $C_{i},C_{i}+w$. We have that $(\mathbb{Z}_{q-1},DevC_{i})$ is a self-dual incidence structure and by Lemma \ref{le1} the number of blocks of $(\mathbb{Z}_{q-1},DevC_{0} \cup DevC_{1})$ in which the points $C_{i}^{\perp},(C_{i}+w)^{\perp}$ appear is, if $q\equiv 1($mod $4)$, \[ \begin{cases} \frac{q-9}{4}+\frac{q-1}{4}=\frac{2q-10}{4} & \text{ if } w \text{ even and } (1-\alpha^{w})^{-1} \in D_{0}^{2},\\
\frac{q-5}{4}+\frac{q-5}{4}=\frac{2q-10}{4} & \text{ if } w \text{ even and } (1-\alpha^{w})^{-1} \in D_{1}^{2},\\
\frac{q-5}{4}+\frac{q-1}{4}=\frac{2q-6}{4} & \text{ otherwise},
\end{cases}\] and if $q\equiv 3($mod $4)$,
\[ \begin{cases} \frac{q-3}{4}+\frac{q-3}{4}=\frac{2q-6}{4} & \text{ if } w \text{ odd and } (1-\alpha^{w})^{-1} \in D_{0}^{2},\\
\frac{q-7}{4}+\frac{q+1}{4}=\frac{2q-6}{4} & \text{ if } w \text{ even and } (1-\alpha^{w})^{-1} \in D_{1}^{2},\\
\frac{q-7}{4}+\frac{q-3}{4}=\frac{2q-10}{4} & \text{ otherwise}.
\end{cases}\]
It is easy to see that the block sizes of the incidence structures $(\mathbb{Z}_{q-1},DevC_{0})$ and $(\mathbb{Z}_{q-1},DevC_{1})$ are $\frac{q-3}{2}$ and $\frac{q-1}{2}$ respectively and that the number of blocks containing a given point in $(\mathbb{Z}_{q-1},DevC_{0})$ is $\frac{2q-6}{4}$. Then the incidence structure $(\mathbb{Z}_{q-1}\cup\{\infty\},Dev^{\infty}C_{0} \cup DevC_{1})$, where $Dev^{\infty}C_{0}$ denotes the blocks of $DevC_{0}$ each modified by adjoining the point $\infty$, is a $2$-adesign.
\end{proof}

\begin{example} With $q=11$  and $C_{i}$ defined as in Theorem \ref{th557} we get that $(\mathbb{Z}_{10},Dev^{\infty}C_{0} \cup DevC_{1})$ is a $2$-$(10,5,3)$ adesign with blocks:
\begin{center}
\begin{tabular}{ccccc}
$\{0,1,3,4,8\}$&$\{1,3,4,6,7\}$&$\{2,4,5,7,8\}$&$\{0,2,3,7,9\}$&$\{3,5,6,8,9\}$ \\
$\{0,1,5,7,8\}$&$\{1,2,4,5,9\}$&$\{0,2,3,5,6\}$&$\{0,4,6,7,9\}$&$\{1,2,6,8,9\}$ \\
$\{2,5,6,7,\infty\}$&$\{0,1,6,9,\infty\}$&$\{3,6,7,8,\infty\}$&$\{1,4,5,6,\infty\}$&$\{0,1,2,7,\infty\}$\\
$\{1,2,3,8,\infty\}$&$\{2,3,4,9,\infty\}$&$\{4,7,8,9,\infty\}$&$\{0,5,8,9,\infty\}$&$\{0,3,4,5,\infty\}$
\end{tabular}\end{center}
\end{example}

The next two constructions will use the following lemmas.

\begin{lemma}\label{le51} (\cite{AND}) Let $p$ be a prime. The number of pairs of consecutive quadratic residues mod $p$ is \[ N(p)=\frac{1}{4}(p-4-(-1)^{\frac{p-1}{2}}) \] and the number of pairs of consecutive quadratic non-residues mod $p$ is \[ N'(p)=\frac{1}{4}(p-2+(-1)^{\frac{p-1}{2}}). \]
\end{lemma}

In the sequel we will sometimes use the following lemma without making reference to it.

\begin{lemma}\label{le50} (\cite{ACH}) Let $p\equiv 1 ($mod $4)$ be a prime. The the set of quadratic residues mod $p$ forms a $(p,\frac{p-1}{2},\frac{p-5}{4},\frac{p-1}{2})$ almost difference set in $\mathbb{Z}_{p}$.
\end{lemma}

\begin{lemma}\label{le52} Let $p \equiv 1 ($mod $4)$ be a prime and $D \subseteq \mathbb{Z}_{p}$ be the set of quadratic residues. Two distinct points $x,y \in D$ occur together in exactly $\frac{p-5}{4}$ translates of $D$ if and only if $x-y$ is a quadratic residue. Dually, $D+x$ and $D+y$ are translates of $D$ with $x-y \in D$ if and only if $|(D+x)\cap(D+y)|=\frac{p-5}{4}$.
\end{lemma}

\begin{proof} Let $x,y \in D$ be distinct. Denote $\frac{p-5}{4}$ by $\lambda$. Without loss of generality we can take $y=1$. Let \[ D, D+\alpha_{1},..., D+\alpha_{\lambda-1} \]be precisely the $\lambda$ translates of $D$ in which $x$ and $1$ appear together. Then \[ x=x_{1}+\alpha_{1}=\cdots=x_{\lambda-1}+\alpha_{\lambda-1} \]for some distinct quadratic residues $x_{1},...,x_{\lambda-1}$ and \[ 1=y_{1}+\alpha_{1}=\cdots=y_{\lambda-1}+\alpha_{\lambda-1} \] for some distinct quadratic residues $y_{1},...,y_{\lambda-1}$. Now suppose that $x-1$ is a quadratic non-residue. Then \[ x-1=x_{1}-y_{1}=\cdots=x_{\lambda-1}-y_{\lambda-1} \] which, by multiplying by appropriate inverses, gives precisely $\lambda$ pairs of consecutive non-residues, these being the {\it only} pairs of consecutive quadratic non-residues. But this contradicts Lemma \ref{le51}, from which we have that the number of pairs of consecutive quadratic non-residues is $\lambda+1$. The condition is necessary and sufficient, and the dual argument follows from the fact that the $2$-adesign $(\mathbb{Z}_{p}$, Dev$D)$ is symmetric.
\end{proof}

We are now ready to construct two more families of $2$-adesigns.

\begin{theorem}\label{th51} Let $p\equiv 1 ($mod $4)$ be a prime greater than $5$, and let $D \subseteq \mathbb{Z}_{p}$ be the set of quadratic residues. Let $\mathcal{B}=\{b\cap D \mid b \in$ Dev$D,b\neq D\}$, and let $\mathcal{B}_{\infty}$ be the set containing all members of $\mathcal{B}$ of size $\frac{p-1}{4}$, as well as all members of $\mathcal{B}$ of size $\frac{p-5}{4}$ modified by adjoining the point $\infty$. Then $(D\cup\{\infty\},\mathcal{B}_{\infty})$ is a $2$-$(\frac{p+1}{2},\frac{p-1}{4},\frac{p-9}{4})$ adesign.
\end{theorem}

\begin{proof} Let $x,y \in D$ be distinct. Denote $\frac{p-5}{4}$ by $\lambda$ and $\frac{p-1}{2}$ by $k$. If $x$ and $y$ appear together in exactly $\lambda$ translates of $D$, then $x$ and $y$ appear together in exactly $\lambda$ blocks in $\mathcal{B}_{\infty}$. Similarly, if $x$ and $y$ appear together in $\lambda+1$ translates of $D$ then $x$ and $y$ appear together in $\lambda+1$ blocks in $\mathcal{B}_{\infty}$. We want to show that $x$ and $\infty$ appear together in exactly $\lambda$ blocks in $\mathcal{B}_{\infty}$. Without loss of generality, we can take $x=1$. There are $k-1$ blocks in $\mathcal{B}_{\infty}$ containing $1$. Let \[ D, D+\alpha_{1},...,D+\alpha_{w} \] be precisely the translates of $D$ containing $1$. By Lemma \ref{le52}, if $|D\cap(D+\alpha_{i})|=\lambda$, then $\alpha_{i}$ is a quadratic residue. If $y+\alpha_{i}=1$ then we have a pair $y,-\alpha_{i}$ of consecutive quadratic residues. By Lemma \ref{le51}, the number of pairs of consecutive quadratic residues is exactly $\lambda$.
\par
To see that there are pairs $x,y \in D$ of distinct points appearing in $\lambda-1$ blocks as well as those appearing in $\lambda$ blocks, suppose that $y_{1},...,y_{k-1}$ be the $k-1$ points in $D-\{1\}$. We can again, without loss of generality, take $x=1$. Suppose that $1$ and $y_{i}$ appear together in exactly $\lambda$ translates of $D$ for each $i$, $1\leq i\leq k-1$. The $y_{i}-1 \in D$ for all $y_{i}$. By Lemma \ref{le51} this gives too many pairs of consecutive quadratic residues, which completes the proof.
\end{proof}

\begin{example} With $p=13$ we apply Theorem \ref{th51} and get that $(D\cup\{\infty\},\mathcal{B}_{\infty})$ is a $2$-$(7,3,1)$ adesign and $\mathcal{B}_{\infty}$ contains the following blocks:
\begin{center}
\begin{tabular}{cccc}
$\{4,10,\infty\}$&$\{3,4,10\}$&$\{1,3,12\}$&$\{4,9,12\}$ \\
$\{4,12,\infty\}$&$\{10,12,\infty\}$&$\{1,3,\infty\}$&$\{4,9,\infty\}$ \\
$\{1,4,9\}$&$\{1,10,12\}$&$\{3,9,10\}$&$\{3,9,\infty\}$
\end{tabular}\end{center}
\end{example}

Let $\mathcal{B}$ and $\mathcal{B}_{\infty}$ be defined as in Theorem \ref{th51}. The second construction is the following.

\begin{theorem}\label{th52} Let $p\equiv 1 ($mod $4)$ be a prime greater than $5$, and let $D \subseteq \mathbb{Z}_{p}$ be the set of quadratic residues. Let $\bar{\mathcal{B}}_{\infty}$ be the set of complements of members of $\mathcal{B}_{\infty}$ in $\mathbb{Z}_{p}\cup\{\infty\}$. Then $(D\cup\{\infty\},\bar{\mathcal{B}}_{\infty})$ is a $2$-$(\frac{p+1}{2},\frac{p+3}{4},\frac{p-5}{4})$ adesign.
\end{theorem}

\begin{proof} Let $x,y \in D\cup\{\infty\}$ be distinct. Denote $\frac{p-5}{4}$ by $\lambda$ and $\frac{p-1}{2}$ by $k$. Suppose $x$ and $y$ appear together in $\lambda$ blocks in $\mathcal{B}_{\infty}$. Then there are $\lambda$ blocks in $\bar{\mathcal{B}}_{\infty}$ not containing $x$ or $y$. Also there are $k-1$ blocks in $\bar{\mathcal{B}}_{\infty}$ not containing $x$ and $k-1$ blocks not containing $y$. Then the number of blocks in $\bar{\mathcal{B}}_{\infty}$ containing $x$ and $y$ is \[
|\bar{\mathcal{B}}_{\infty}|-|\{b\in\bar{\mathcal{B}}_{\infty}\mid x \notin b\}\cup\{b \in \bar{\mathcal{B}}_{\infty} \mid y \notin b\}|+|\{b\in\bar{\mathcal{B}}_{\infty}\mid x,y \notin b\}|\]which is easily seen to be $\lambda+1$. A similar calculation shows that if $x$ and $y$ appear together in $\lambda-1$ blocks in $\mathcal{B}_{\infty}$ then $x$ and $y$ appear together in $\lambda$ blocks $\bar{\mathcal{B}}_{\infty}$.
\end{proof}

\begin{example} With $p=13$ we apply Theorem \ref{th52} and get that $(D\cup\{\infty\},\bar{\mathcal{B}}_{\infty})$ is a $2$-$(7,4,2)$ adesign and $\bar{\mathcal{B}}_{\infty}$ contains the following blocks:
\begin{center}
\begin{tabular}{cccc}
$\{1,3,9,12\}$&$\{4,9,10,\infty\}$&$\{1,3,9,10\}$&$\{4,9,10,12\}$ \\
$\{3,9,10,\infty\}$&$\{1,4,12,\infty\}$&$\{1,9,12,\infty\}$&$\{1,3,10,\infty\}$ \\
$\{1,3,4,9\}$&$\{3,4,10,12\}$&$\{3,4,9,\infty\}$&$\{1,4,10,12\}$
\end{tabular}\end{center}
\end{example}

\begin{remark} Interestingly, the $2$-adesigns obtained by Theorem \ref{th52} in fact meet the packing bound given in Theorem \ref{th888}, i.e., they are a family of maximal $(\lambda+1)$-packings.
\end{remark}

\section{Constructions of $2$-adesigns that are Almost Difference Families}\label{sec4}

Almost difference families were studied by Ding et al. in \cite{ADF}. Suppsoe $G$ is a finite Abelian group f order $v$ in which the identity element is denoted ``0''. Let $k$ and $\lambda$ be positive integers such that $2\leq k<v$. A $(v,k,\lambda)$ {\it difference family} in $G$ is a collection of subsets $D_{0},...,D_{l}$ of $G$ such that
\begin{enumerate}
\item $\left|D_{i}\right|=k$ for all $i,0\leq i\leq l$,
\item the multiset union $\cup_{i=1}^{l}\{x-y \mid x,y \in D_{i}, x \neq y\}$ contains each member of $G-\{0\}$ either $\lambda$ times, \end{enumerate} and a $(v,k,\lambda,t)$ {\it almost difference family} is defined similarly only the multiset union $\cup_{i=1}^{l}\{x-y \mid x,y \in D_{i}, x \neq y\}$ contains $t$ members of $G-\{0\}$ with multiplicity $\lambda$ and $v-t-1$ members of $G$ with multiplicity $\lambda+1$.

It is trivial that an almost difference family is a $2$-adesign. All of the $2$-adesigns in this section are also almost difference families, however, our treatment will still be in terms of $2$-adesigns.

Our next two constructions make use of quadratic residues. We will need the following lemma \cite{D}.

\begin{lemma}\label{le4} Let $q=4f+1=x^{2}+4y^{2}$ be a prime power with $x,y \in \mathbb{Z}$ and $x \equiv 1$ (mod 4) (here, $y$ is two-valued depending on the choice of the primitive root $\alpha$ defining the cyclotomic classes). The five distinct cyclotomic numbers of order four for odd $f$ are
\begin{eqnarray*}
(0,0) & = & (2,2) = (2,0) = \frac{q-7+2x}{16}    \\
(0,1) & = & (1,3) = (3,2) = \frac{q+1+2x-8y}{16} \\
(1,2) & = & (0,3) = (3,1) = \frac{q+1+2x+8y}{16} \\
(0,2) & = & \frac{q+1-6x}{16}                    \\
\text{all others} & = & \frac{q-3-2x}{16}
\end{eqnarray*} and those for $f$ even are
\begin{eqnarray*}
(0,0) & = & \frac{q-11-6x}{16}    \\
(0,1) & = & (1,0) = (3,3) = \frac{q-3+2x+8y}{16} \\
(0,2) & = & (2,0) = (2,2) = \frac{q-3+2x}{16} \\
(0,3) & = & (3,0) = (1,1) = \frac{q-3+2x-8y}{16} \\
\text{all others} & = & \frac{q+1-2x}{16}
\end{eqnarray*}
\end{lemma}

When computing difference levels of a subset $C$ of a group $G$, it is sometimes convenient to use the difference function which is defined as $d(w) = \left|C \cap (C+w)\right|$ where $C+w$ denotes the set $\{c+w \mid c\in C\}$. We are now ready to give our first construction of a $2$-adesign that is a difference family.
\begin{theorem}\label{th555} Let $q=4f+1=x^{2}+4y^{2}$ be a prime power with $f$ odd. Let $C_{0} = D_{0}^{4} \cup D_{1}^{4}, C_{1}=D_{0}^{4} \cup D_{2}^{4}$, and $C_{2}=D_{0}^{4} \cup D_{3}^{4}$. Then $(\mathbb{F}_{q},DevC_{0} \cup DevC_{1} \cup DevC_{2})$ is a $2$-$(q,\frac{q-1}{2},\frac{3q-11}{4})$ adesign.
\end{theorem}

\begin{proof} Let $w^{-1} \in D_{h}^{4}$. First we let $C$ denote $D_{i}^{4} \cup D_{i+1}^{4}$. Then when we expand $\left| C \cap (C+w)\right|$ we get \[\left| D_{i+h}^{2} \cap (D_{i+h}^{2}+1) \right| + \left| D_{i+h}^{4} \cap (D_{i+h+1}^{4}+1) \right| + \left| D_{i+h+1}^{4} \cap (D_{i+h}^{4}+1) \right| +\left| D_{i+h+1}^{4} \cap (D_{i+h+1}^{4}+1) \right|\]

whence

\begin{eqnarray*}
\left| C \cap (C+w)\right| & = &(i+h,i+h) + (i+h,i+h+1) + (i+h+1,i+h) +(i+h+1,i+h+1) \\
                           & = & \begin{cases}
                                 \frac{q-2y-3}{4} & \text{for } i = 0 \text{ and } h=0 \text{ or } 2, \\
                                 \frac{q+2y-3}{4} & \text{for } i = 0 \text{ and } h=1 \text{ or } 3,\\
                                 \frac{q-2y-3}{4} & \text{for } i = 3 \text{ and } h=0 \text{ or } 2,\\
                                 \frac{q+2y-3}{4} & \text{for } i = 3 \text{ and } h=1 \text{ or } 3.
                                 \end{cases} \quad (\text{by Lemmas \ref{le3} and ~\ref{le4}})
\end{eqnarray*} We also have \begin{eqnarray*}
\left| C_{1} \cap (C_{1}+w)\right| & = & \begin{cases}\
\frac{q-5}{4} & \text{for } h = 0 \text{ or } 2, \\
~\frac{q-1}{4} & \text{for } h = 1 \text{ or } 3. \end{cases}
\end{eqnarray*}

Now consider the incidence structures $(\mathbb{F}_{q},DevC_{i})$ for $i=0,1,2$. Let $C_{i}^{\perp},(C_{i}+w)^{\perp}$ denote the points of the dual structures $(DevC_{i},\mathbb{F}_{q})$ corresponding to the blocks $C_{i},C_{i}+w$. We have that $(\mathbb{F}_{q},DevC_{i})$ is a self-dual incidence structure and by Lemma \ref{le1} the number of blocks of $(\mathbb{F}_{q},DevC_{0} \cup DevC_{1} \cup DevC_{2})$ which the points $C_{i}^{\perp},(C_{i}+w)^{\perp}$ appear in is \[ \begin{cases} \frac{3p-11}{4} & \text{ if } w^{-1} \in D_{0}^{4}\cup D_{2}^{4},\\
\frac{3p-7}{4}  & \text{ if } w^{-1} \in D_{1}^{4}\cup D_{3}^{4}.
\end{cases} \]
\end{proof}

Another construction is the following.

\begin{theorem}\label{th556} Let $q=4f+1=x^{2}+4y^{2}$ be a prime power with $f$ even. Then $(\mathbb{F}_{q},DevD_{0}^{4} \cup DevD_{2}^{4})$ is a $2$-$(q,\frac{q-1}{4},\frac{q-7-2x}{8})$ adesign.
\end{theorem}

\begin{proof} We have, by Lemma \ref{le4}, \begin{eqnarray*}
\left| D_{i}^{4} \cap (D_{i}^{4}+w)\right| & = & \left| D_{h}^{4} \cap (D_{h}^{4}+1)\right| \\
                                          & = & (i+h,i+h) \\
                                          & = & \begin{cases}
                                          \frac{q-11-6x}{16} & \text{ if } h = 0,i=0 \text{ or } h = 2,i=2, \\
                                 \frac{q-3+2x-8y}{16} & \text{ if } h = 1,i=0\text{ or } h = 2,i=2,\\
                                 \frac{q-3+2x}{16} & \text{ if } h = 2,i=0\text{ or } h = 3,i=2,\\
                                 \frac{q-3+2x+8y}{16} & \text{for } h = 3,i=0\text{ or } h = 0,i=2.\\
                                 \end{cases}\end{eqnarray*}

Now consider the incidence structures $(\mathbb{F}_{q},DevD_{i}^{4})$ for $i=0,2$. Let $C_{i}^{\perp},(C_{i}+w)^{\perp}$ denote the points of the dual structures $(DevD_{i}^{4},\mathbb{F}_{q})$ corresponding to the blocks $C_{i},C_{i}+w$. We have that $(\mathbb{F}_{q},DevC_{i})$ is a self-dual incidence structure and by Lemma \ref{le1} the number of blocks of $(\mathbb{F}_{q},DevD_{0}^{4} \cup DevD_{2}^{4})$ which the points $C_{i}^{\perp},(C_{i}+w)^{\perp}$ appear in is \[ \begin{cases} \frac{2q-14-4x}{16} & \text{ if } w^{-1} \in D_{0}^{4}\cup D_{2}^{4},\\
\frac{2q-6+4x}{16}  & \text{ if } w^{-1} \in D_{1}^{4}\cup D_{3}^{4}.\end{cases}\] Thus, we have $(\mathbb{F}_{q},DevD_{0}^{4} \cup DevD_{2}^{4})$ is a 2-adesign whenever $x=1$, or $-3$.
\end{proof}

We close this section with yet a few more constructions. Now let $q$ be an odd prime power, and $C \subseteq \mathbb{F}_{q}$. According to \cite{NOW}, if \begin{enumerate}
\item $C=D_{i}^{4}\cup D_{i+1}^4$, $q\equiv5($mod $8)$ and $q=s^{2}+4$ with $s\equiv 1($mod $4)$, or
\item $C=D_{0}^{8}\cup D_{1}^{8}\cup D_{2}^{8}\cup D_{5}^{8}$, $q=l^{2}$ where $l$ is a prime power of form $l=t^{2}+2\equiv 3($mod $8)$, or
\item $C=\cup_{i\in I}D_{i}^{\sqrt{q}+1}$ where $I \subseteq \{0,1,..,\sqrt{q}\}$ with $\left|I\right|=\frac{\sqrt{q}+1}{2}$ and $q=l^{2}$ for some prime power $l$,
\end{enumerate} then $C$ is a $(q,\frac{q-1}{2},\frac{q-5}{4},\frac{q-1}{2})$ almost difference set in $\mathbb{F}_{q}$.

It is easy to show, also, that if $q$ is an odd prime power, $(\mathbb{F}_{q},DevD_{0}^{2}\cup DevD_{1}^{2})$ is a $2$-$(q,\frac{q-1}{2},\frac{2q-6}{4})$ design. We then have the following.

\begin{theorem} \label{th558}Let $q$ be an odd prime power, and $C \subseteq \mathbb{F}_{q}$. If
\begin{enumerate}
\item $C=D_{i}^{4}\cup D_{i+1}^{4}$, $q\equiv5($mod $8)$ and $q=s^{2}+4$ with $s\equiv 1($mod $4)$, or
\item $C=D_{0}^{8}\cup D_{1}^{8}\cup D_{2}^{8}\cup D_{5}^{8}$, $q=l^{2}$ where $l$ is a prime power of form $l=t^{2}+2\equiv 3($mod $8)$, or
\item $C=\cup_{i\in I}D_{i}^{\sqrt{q}+1}$ where $I \subseteq \{0,1,..,\sqrt{q}\}$ with $\left|I\right|=\frac{\sqrt{q}+1}{2}$, $I$ contains both even and odd numbers, and $q=l^{2}$ for some prime power $l$,
\end{enumerate} then $(\mathbb{F}_{q},DevD_{0}^{2}\cup DevD_{1}^{2}\cup DevC)$ is a $2$-$(q,\frac{q-1}{2},\frac{3q-11}{4})$ adesign.
\end{theorem}

\section{Constructions of $2$-adesigns from Symmetric Designs}\label{sec4.5}

Let $(V,\mathcal{B})$ be an incidence structure with $\left|\mathcal{B}\right|=b$. The numbers of blocks in which given single points appear (called the {\it replication numbers}) become the block sizes of the dual $(V,\mathcal{B})^{\perp}$, and the intersection numbers among pairs of blocks become the numbers of blocks of $(V,\mathcal{B})^{\perp}$ in which any two points appear. Then the following is clear.

\begin{lemma}\label{le559} Let $(V,\mathcal{B})$ be an incidence structure with $\left|V\right|=v$, and in which the replication numbers are a constant $k$ and the intersection numbers among pairs of blocks are integers $\lambda$ and $\lambda+1$. Then $(V,\mathcal{B})^{\perp}$ is a $2$-$(b,k,\lambda)$ adesign.
\end{lemma}

\begin{remark} The dual of a quasi-symmetric design whose intersection numbers $x,y$ are such that $y-x=1$ is always a $2$-adesign.
\end{remark}

We will use the following lemma which is actually a trivial construction in itself.

\begin{lemma}\label{le560} Let $(V,\mathcal{B})$ be a symmetric $2$-$(v,k,\lambda)$ design. Let $\mathbf{b}_{1},...,\mathbf{b}_{k}$ be any $k$ blocks in $\mathcal{B}$. Let ``$\infty$'' denote a point. Let $\mathcal{B}'$ denote the blocks of $\mathcal{B}$ modified by adjoining the point ``$\infty$'' to each of $\mathbf{b}_{1},...,\mathbf{b}_{k}$. Then $(V,\mathcal{B}')^{\perp}$ is a $2$-$(v,k,\lambda)$ adesign.
\end{lemma}

\begin{proof} The replication numbers in the incidence structure $(V,\mathcal{B}')$ are all $k$, and the intersection numbers among pairs of blocks in $\mathcal{B}'$ are $\lambda$ and $\lambda+1$. The result follows from Lemma \ref{le559}.
\end{proof}

Note that the number of times which Lemma \ref{le560} can be applied to any given symmetric $2$-$(v,k,\lambda)$ design is $\lfloor \frac{v}{k} \rfloor$.

The following theorem gives another construction.
\begin{theorem}\label{th999} Let $(V,\mathcal{B})$ be a symmetric $2$-$(v,k,\lambda)$ design. Let $\mathbf{b}=\{b_{1},...,b_{k}\}$ be a block. Suppose that $\mathbf{b}_{1},...,\mathbf{b}_{k}$ are $k$ blocks not equal to $\mathbf{b}$ such that \begin{enumerate}
\item $b_{i} \not\in \mathbf{b}_{i}$ for all $i,1\leq i \leq k$, and such that
\item $b_{j} \in \mathbf{b}_{l}$ implies $b_{l} \not\in \mathbf{b}_{j}$ for all $j \neq l,1\leq j,l\leq k$.\end{enumerate}
Let $\mathcal{B}'$ denote the blocks of $\mathcal{B}$ modified by adjoining the point $b_{i}$ to the block $\mathbf{b}_{i}$ for all $i,1\leq i \leq k$, and then removing the block $\mathbf{b}$. Then $(V,\mathcal{B}')^{\perp}$ is a $2$-$(v,k,\lambda)$ adesign.
\end{theorem}

\begin{proof} It is easy to see that the replication numbers of $(V,\mathcal{B}')$ are all $k$. The second condition in the statement ensures that the intersection numbers among pairs of blocks of $\mathcal{B}'$ are either $\lambda$ or $\lambda+1$. The result then follows from Lemma \ref{le559}.
\end{proof}

Next, we show how to construct almost difference sets from planar difference sets. The following constructions are not optimal but, for certain dimensions, give the best known value for $d_{1}$. A $(v,k,\lambda)$ difference set is called planar if $\lambda=1$. It is easy to show that, given a planar difference set $D$ in an (additive) Abelian group $G$ of order $v$, if we choose any $a_{0}\in G-D$ such that $2a_{0}$ cannot be written as the sum of two distinct members of $D$, then $D\cup \{a_{0}\}$ will be an almost difference set with $\lambda=1$. This is simply due to the fact that, because of the way we chose $a_{0}$, we cannot have $a_{0}-a=b-a_{0}$ for any $a,b \in D$, thereby forcing each member of $G$ to appear as a difference of two distinct members of $D\cup \{a_{0}\}$ only one or two times.
\par
Again, let $D$ be a $(v,k,1)$ difference set in an Abelian group $G$ of order $v$. Also let $\kappa : G\rightarrow \mathbb{Z}_{2} \times G$ by $x \mapsto (0,x)$. Suppose $a_{0},...,a_{s-1}\in G$ are such that the differences $(1,\tau)$ in $\kappa(D)\cup\{(1,a_{0}),...,(1,a_{s-1})\}$ cover $\{1\}\times G$ each having multiplicity at most $2$, that exactly one of the $a_{i}$s is a member of $D$, and twice any $a_{i}$ is not the sum of two other distinct $a_{i}$s. If there is at least one difference in $\kappa(D)\cup\{(1,a_{0}),...,(1,a_{s-1})\}$ having multiplicity $1$, then since the difference $(1,0)$ occurs exactly twice (because exactly one of the $a_{i}$s is in $D$), we have both $1$ and $2$ occurring as multiplicities. No difference can occur with multiplicity greater than $2$ since $G$ is planar and twice any $a_{i}$ is not the sum of two other distinct $a_{i}$s.  We also have the differences in $\kappa(D)\cup\{(1,a_{0}),...,(1,a_{s-1})\}$ covering $\mathbb{Z}_{2}\times G$: the differences $(0,\tau)$ cover $\{0\}\times G$ due to $G$ being a planar difference set and we have assumed that the differences $(1,\tau)$ cover $\{1\}\times G$. This discussion is summarized in the following.
\begin{theorem}\label{th61} Let $D$ be a $(v,k,1)$ difference set in an (additive) Abelian group $G$. Suppose $a_{0},...,a_{s-1}\in G$ are such that the differences $(1,\tau)$ in $\kappa(D)\cup\{(1,a_{0}),...,(1,a_{s-1})\}$ cover $\{1\}\times G$ each having multiplicity at most $2$, that exactly one of the $a_{i}$s is a member of $D$, and twice any $a_{i}$ is not the sum of two other distinct $a_{i}$s. If there is at least one difference in $\kappa(D)\cup\{(1,a_{0}),...,(1,a_{s-1})\}$ having multiplicity $1$ then $\kappa(D)\cup\{(1,a_{0}),...,(1,a_{s-1})\}$ is a $(2v,k+s,1,t)$ almost difference set in $\mathbb{Z}_{2}\times G$. The resulting symmetric $2$-adesign $(\mathbb{Z}_{2}\times G,$ Dev$(\kappa(D)\cup\{(1,a_{0}),...,(1,a_{s-1})\}))$ has parameters $(2v,k+s,1)$.
\end{theorem}

\begin{example}\label{ex1} Consider the Singer difference set $\{1,2,4\}$ in $\mathbb{Z}_{7}$. With $a_{0}=0$ we have $2a_{0}$ is not the sum of two distinct members of $D$, and $\kappa(D)\cup\{(1,0)\}$ is a $(14,4,0,1)$ almost difference set in $\mathbb{Z}_{14}$. With $a_{1}=1$ we have $\kappa(D)\cup\{(1,0),(1,1)\}$ is a $(14,5,1,6)$ almost difference set in $\mathbb{Z}_{14}$
\end{example}

\begin{example} Consider the Singer difference set $D=\{0,1,5,11\}$ in $\mathbb{Z}_{13}$. With $a_{0}=10$, we have $2a_{0}$ is not the sum of two distinct members of $D$, and it is easily checked that $\kappa(D)\cup \{(1,10)\}$ is a $(26,5,0,5)$ almost difference set in $\mathbb{Z}_{26}$. With $a_{1}=11$ we have that $\kappa(D)\cup\{(1,a_{0}),(1,a_{1})\}$ is a $(26,6,1,11)$ almost difference set.
\end{example}

\begin{example} Now consider the Singer difference set $D=\{0,3,13,15,20\}$ in $\mathbb{Z}_{21}$. We have $\{9,13,16\}$ are such that the differences $(1,\tau)$ cover $\{1\}\times \mathbb{Z}_{21}$ with multiplicities no more than $2$ and that $13$ is the only member that is also in $D$. It is also easy to see that the difference $(1,9)$ can only occur as the difference $(1,9)-(0,0)$. Thus we have $\kappa(D)\cup \{(1,9),(1,13),(1,16)\}$ is a $(42,8,1,16)$ almost difference set.
\end{example}

\section{Constructions of $3$-adesigns}\label{sec5}

In this section we will give two constructions each of which produce infinitely many $3$-adesigns.

Our first constructions makes use of quadratic residues.

\begin{theorem}\label{th561} Let $q\equiv3($mod $4)$ be an odd prime power. Then
$(\mathbb{F}_{q},DevD_{0}^{2}\cup DevD_{1}^{2})$ is a $3$-$(q,\frac{q-1}{2},\frac{q-7}{4})$ adesign.
\end{theorem}
\begin{proof}
Denote $\frac{q-1}{2}$ by $k$ and $\frac{q-3}{4}$ by $\lambda'$. Let $x,y,z\in \mathbb{F}_{q}$ be arbitrary. To count the number of blocks in which $x,y,z$ appear together, we first count the number of blocks of $DevD_{0}^{2}\cup Dev(D_{1}^{2}\cup \{0\})$ in which $x,y,z$ appear together. Suppose that the three points $x,y,z$ appear in $\mu$ blocks in $DevD_{0}^{2}$. Using the fact that $(\mathbb{F}_{q},DevD_{0}^{2})$ is a $2$-$(q,k,\lambda')$ design, a simple counting argument gives that there are $q-3k+3\lambda'-\mu$ blocks in $Dev\overline{D_{0}^{2}}:=Dev(D_{1}^{2}\cup\{0\})$ containing $x,y,z$. Thus, there are $q-3k+3\lambda'=\lambda'$ blocks in $DevD_{0}^{2}\cup Dev\overline{D_{0}^{2}}$ containing $x,y,z$. Since $w\in D_{1}^{2}\cup\{0\}+w$ for all $w\in \mathbb{F}_{q}$, we want to know how many of the $q-3k+3\lambda'-\mu$ blocks in $Dev\overline{D_{0}^{2}}$ are also in $\{\overline{D_{0}^{2}}+x,\overline{D_{0}^{2}}+y,\overline{D_{0}^{2}}+z\}$. Without loss of generality suppose that both $\overline{D_{0}^{2}}+x$ and $\overline{D_{0}^{2}}+y$ contain the three points $x,y,z$. Then we must have $y-x,z-x \not\in D_{0}^{2}$ and $x-y,z-y \not\in D_{0}^{2}$. But this would imply that $x-y,y-x \in D_{1}^{2}$ where both $x-y$ and $y-x$ are nonzero. But this is impossible as the additive inverse of any member of $D_{1}^{2}$ cannot also be a member whenever $q\equiv3($mod $4)$. Then no more than one of the blocks $\overline{D_{0}^{2}}+x,\overline{D_{0}^{2}}+y,\overline{D_{0}^{2}}+z$ can contain all three of $x,y,z$. We now need to show that there are two different $3$-levels, i.e. that $(\mathbb{F}_{q},DevD_{0}^{2}\cup DevD_{1}^{2})$ is not a $3$-design, but a $3$-adesign. To show this we assume that $(\mathbb{F}_{q},DevD_{0}^{2}\cup DevD_{1}^{2})$ is a $3$-$(q,k,\lambda)$-design for some $\lambda$. Then the number of blocks must be given by $\lambda\frac{\binom{q}{3}}{\binom{k}{3}}$. The only choices for $\lambda$ are $\lambda'$ or $\lambda'-1$. If $\lambda=\lambda'$ then we get that $q-5=q-4$. If $\lambda=\lambda'-1$ then we get that $(q-3)(q-5)=(q-7)(q-2)$. Either way we get a contradiction, which completes the proof.\end{proof}

\begin{example}\label{ex562} With $q=11$ we apply Theorem \ref{th561} and get that $(\mathbb{Z}_{11},DevD_{0}^{2}\cup DevD_{1}^{2})$ is a $3$-$(11,5,1)$ adesign with blocks:
\begin{center}
\begin{tabular}{cccccc}
$\{1,3,4,5,9\}$&$\{2,4,5,6,10\}$&$\{0,3,5,6,7\}$&$\{1,4,6,7,8\}$&$\{2,5,7,8,9\}$&$\{0,4,5,6,8\}$ \\
$\{3,6,8,9,10\}$&$\{0,4,7,9,10\}$&$\{0,1,5,8,10\}$&$\{0,1,2,6,9\}$&$\{1,2,3,7,10\}$&$\{1,5,6,7,9\}$ \\
$\{0,2,3,4,8\}$&$\{2,6,7,8,10\}$&$\{0,3,7,8,9\}$&$\{1,4,8,9,10\}$&$\{0,2,5,9,10\}$& \\
$\{0,1,3,6,10\}$&$\{0,1,2,4,7\}$&$\{1,2,3,5,8\}$&$\{2,3,4,6,9\}$&$\{3,4,5,7,10\}$& \\
\end{tabular}\end{center}
\end{example}

Our second construction is related to graphs, though it is simple enough to avoid graph-theoretical preliminaries.

\begin{theorem}\label{th564}
Let $n$ $(\geq 7)$ be an odd integer not divisible by 3. Consider, for fixed $a\in\mathbb{Z}_{n}$, all pairs $\lbrace a-i\pmod n, a+i\pmod n \rbrace$ for $i=1,\cdots,\frac{n-1}{2}$. The union of any two distinct pairs gives a block consisting of four points. Denote, for fixed $a\in\mathbb{Z}_{n}$, the set of all blocks obtained in this way by $B_{a}$. Then $(\mathbb{Z}_{n},\cup_{a\in\mathbb{Z}_{n}}B_{a})$ is a $3$-$(n,4,2)$ adesign.
\end{theorem}
\begin{proof}
Arrange all the points in a circle as is shown in the graph below. For any three points $x, y, z \in \mathbb{Z}_{n}$, denote $|x-y|$, $|x-z|$, $|y-z|$ by $d_{xy}$, $d_{xz}$, $d_{yz}$ respectively.

Since $n$ is not divisible by $3$, $d_{xy}=d_{xz}=d_{yz}$ cannot happen. Then suppose two of them are equal. Without loss of generality, suppose $d_{xz}=d_{yz}$. Then when $x$ and $y$ are in a pair, $z$ must be the fixed point so that there is no block containing all three of $x,y$ and $z$. When $x$ and $z$ are in a pair or $y$ and $z$ are in pair, we can find exactly one block containing the three points in each case. If $d_{xy},d_{xz}$ and $d_{yz}$ are distinct, then we can find one block containing these three points when any two points are in pair, in which case we have three blocks containing these three points together.
\begin{center}
\begin{tikzpicture}
%\label{graph1}
\draw  (0,0) circle [radius=2];
\draw [fill] (0, 0) circle [radius=0.05];
\draw [fill] (-1.6,1.2) circle [radius=0.05];
\draw [fill] (1.6,1.2) circle [radius=0.05];
\draw [fill] (0,-2) circle [radius=0.05];
\draw [fill] (0,2) circle [radius=0.02];
\draw [fill] (0.5,1.94) circle [radius=0.02];
\draw [fill] (-0.5,1.94) circle [radius=0.02];
\draw [fill] (1,1.73) circle [radius=0.02];

\draw [-] (-1.6,1.2) -- (1.6,1.2);
\draw [-] (-1.6,1.2) -- (0,-2);
\draw [-] (0,-2) -- (1.6,1.2);

\node [left] at (-1.8,1.2) {$x$};
\node [left] at (2.2,1.2) {$y$};
\node [below] at (0,-2.2) {$z$};
\node [above] at (0,2) {\tiny$0$};
\node [above] at (0.5,1.94) {\tiny$1$};
\node [above,rotate=13] at (-0.5,1.94) {\tiny$n-1$};
\node [above] at (1,1.73) {\tiny$2$};
\node [right,rotate=138] at (1.65,1.45) {\small$\cdots$};

\node [above] at (0,1.2) {$d_{xy}$};
\node [left] at (-0.9,0) {$d_{xz}$};
\node [right] at (1,0) {$d_{yz}$};
\end{tikzpicture}
\end{center}
\end{proof}

\begin{example} With $n=7$ we apply Theorem \ref{th564} and get that $(\mathbb{Z}_{7},\cup_{a\in\mathbb{Z}_{7}}B_{a})$ is a $3$-$(7,4,2)$ adesign with blocks:
\begin{center}
\begin{tabular}{ccccccc}
$\{1,7,2,6\}$&$\{1,7,3,5\}$&$\{2,6,3,5\}$&$\{7,6,1,5\}$&$\{7,6,2,4\}$&$\{1,5,2,4\}$&$\{1,4,2,3\}$ \\
$\{1,3,7,4\}$&$\{1,3,6,5\}$&$\{7,4,6,5\}$&$\{7,2,6,3\}$&$\{7,2,5,4\}$&$\{6,3,5,4\}$&$\{1,2,7,3\}$ \\
$\{1,6,2,5\}$&$\{1,6,3,4\}$&$\{2,5,3,4\}$&$\{7,5,1,4\}$&$\{7,5,2,3\}$&$\{7,3,6,4\}$&$\{1,2,6,4\}$
\end{tabular}\end{center}
\end{example}

Let $(V,\mathcal{B})$ be an incidence structure. Let $p\in V$, and define $\mathcal{B}_{p}=\{\mathcal{B}-\{p\}\mid\mathcal{B}\in\mathcal{B}\text{ and }p\in\mathcal{B}\}$. We call the incidence structure $(V-\{p\},\mathcal{B}_{p})$ the {\it contraction} of $(V,\mathcal{B})$ at $p$. It is clear that contracting at points of a $3$-adesign will give a $2$-adesign as long as not all $3$-sets of points occur in the same number of blocks of the contraction.

\begin{example}\label{ex563} The contraction at the point $p=1$ of the $3$-$(11,5,1)$ adesign in Example \ref{ex562} is a symmetric $2$-$(10,4,1)$ adesign with the ten blocks:
\begin{center}
\begin{tabular}{ccccc}
$\{3,4,5,9\}$&$\{4,6,7,8\}$&$\{0,5,8,10\}$&$\{0,2,6,9\}$&$\{2,3,7,10\}$ \\
$\{4,8,9,10\}$&$\{0,3,6,10\}$&$\{0,2,4,7\}$&$\{2,3,5,8\}$&$\{5,6,7,9\}$
\end{tabular}\end{center}
\end{example}

\begin{remark}\label{re564} Interestingly, a contraction at any point of the incidence structure $(\mathbb{F}_{q},DevD_{0}^{2}\cup DevD_{1}^{2})$ from Theorem \ref{th561} gives a symmetric $2$-$(q-1,\frac{q-3}{2},\frac{q-7}{4})$ adesign which is not the development of any almost difference set.
\end{remark}

We close this section with a table describing new $t$-$(v,k,\lambda)$ adesigns constructed in this paper for $v\leq18$.
\newpage
\begin{longtable}{|c|c|c|c|}
\caption{Parameters of new t-adesigns}\\
  \hline
  % after \\: \hline or \cline{col1-col2} \cline{col3-col4} ...
  $t$-adesign ref & $t$ & $(v,k,\lambda)$& no. of blocks \\
  \hline
  Theorem \ref{th557} & 2 & $(q,\frac{q-1}{2},\frac{q-5}{2})$ & $2q-2$ \\\hline
  Theorem \ref{th51} & 2 & $(\frac{p+1}{2},\frac{p-1}{4},\frac{p-9}{4})$ & $p-1$   \\\hline
  Theorem \ref{th52}$^{*}$ & 2 & $(\frac{p+1}{2},\frac{p+3}{4},\frac{p-5}{4})$ & $p-1$  \\\hline
  Theorem \ref{th555} & 2 & $(q,\frac{q-1}{2},\frac{3q-11}{4})$ & $3q$  \\\hline
  Theorem \ref{th556} & 2 & $(q,\frac{q-1}{4},\frac{q-7-2x}{8})$ & $2q$  \\\hline
  Theorem \ref{th558} & 2 & $(q,\frac{q-1}{2},\frac{3q-11}{4})$ & $3q$  \\\hline
  Remark \ref{re564}  & 2 & $(q-1,\frac{q-3}{2},\frac{q-7}{4})$ & $2q$\\\hline
  Theorem \ref{th561} & 3 & $(q,\frac{q-1}{2},\frac{q-7}{4})$ & $2q$   \\\hline
  Theorem \ref{th564} & 3 & $(n,4,2)$ & $\frac{n(n-1)(n-3)}{8}$  \\
  \hline
  \multicolumn{4}{l}{{\footnotesize * This $2$-adesign meets the packing bound given in Theorem \ref{th888}.}}\\
\end{longtable}

\section{Related Codes}\label{sec5.5}

A linear binary {\it code} $C$ of {\it length} $n$ and {\it dimension} $k$ (or simply an $\left[n,k\right]$ code), is a $k$-dimensional linear subspace of the $n$-dimensional binary vector space $\mathbb{F}_{2}^{n}$. The {\it dual} $C^{\perp}$ of an $\left[n,k\right]$ code $C$ is the $\left[n,n-k\right]$ code that is the orthogonal space of $C$ with respect to the inner product of the binary field. Any basis of $C$ is called a {\it generator matrix} of $C$, and any basis of $C^{\perp}$ is called a {\it parity check matrix} of $C$. The Hamming distance between two vectors $x=(x_{1},...,x_{n})$ and $y=(y_{1},...,y_{n})$ is the number of indices $i$ such that $x_{i} \neq y_{i}$. The Hamming {\it weight} of a vector is the number of its nonzero coordinates. The minimum distance $d$ of a code is smallest possible distance between pairs of distinct codewords. An $\left[n,k\right]$ code $C$ is {\it self-orthogonal} if $C \subseteq C^{\perp}$.
\par
In order to discuss the dimensions of some of the codes from adesigns we will need the following results concerning coverings and packings.
\par
Let $v$, $k$, $\lambda$ be positive integers and let $(V,\mathcal{B})$ be a pair where $V$ is a $v$-set of points and $\mathcal{B}$ is a collection of $k$-subsets of $V$, called blocks. If each pair of points occur together in at least $\lambda$ blocks, then $(V,\mathcal{B})$ is a $(v,k,\lambda)$-{\it covering}. If each pair of points occur together in at most $\lambda$ blocks, then $(V,\mathcal{B})$ is a $(v,k,\lambda)$-{\it packing}.

Denote the incidence matrix of a $(v,k,\lambda)$-covering resp. -packing by $M_{c}$ resp. $M_{p}$. Also note that if $b$ is the number of blocks in $\mathcal{B}$ then we always have \[ b\geq rank(M)\geq rank(MM^{T}), \] where $M$ is either of $M_{c}$ or $M_{p}$. Then we have the following two theorems (\cite{Covering}):
\begin{theorem}
Let $v$, $k$, $\lambda$ be positive integers such that $3 \leq k <v$, and let $r_{1}$ and $d_{1}$ be the integers such that $\lambda(v-1)=r_{1}(k-1)-d_{1}$ and $0\leq d_{1}<k-1$. If $d_{1}<r_{1}-\lambda$, then
\begin{equation}
rank(M_{c}M_{c}^{T})\geq\Big\lceil\frac{v(r_{1}+1)}{k+1}\Big\rceil.
\end{equation}
\end{theorem}

\begin{theorem}\label{th888}
Let $v$, $k$, $\lambda$ be positive integers such that $3 \leq k <v$, and let $r_{2}$ and $d_{2}$ be the integers such that $\lambda(v-1)=r_{2}(k-1)+d_{2}$ and $0\leq d_{2}<k-1$. Let $b$ be the number of blocks in any $(v,k,\lambda)$-packing. If $d_{2}<r_{2}-\lambda$, then
\begin{equation}
b\leq\Big\lfloor\frac{v(r_{2}-1)}{k-1}\Big\rfloor.
\end{equation}
\end{theorem}

\begin{longtable}{|c|c|c|c|}
\caption{The bounds of the dimension of the codes from some known families of ADS's and adesigns.}\\
\hline
ADS ref & parameters & $d_{1}$  & lower bound \\
\hline
\cite{NOW} Theorem 45 & $(q,\frac{q-1}{2},\frac{q-5}{4},\frac{q-1}{2})$ & 1   & $q-1$  \\
\hline
 \cite{NOW} Theorem 45& $(q,\frac{q-1}{4},\frac{q-13}{16},\frac{q-1}{2}))$ & 2  & $q-7$  \\
\hline
 \cite{NOW} Theorem 45& $(q,\frac{q+3}{4},\frac{q-5}{16},\frac{q-1}{2})$ & 0  & $q-7$  \\
\hline
\cite{ACH} Theorem 22& $(4p,2p-1,p-2,p-1)$ & 3  & $4p-3$  \\
\hline
\cite{ACH} Theorem 22& $(4(2^{t}-1),2^{t+1}-3,2^{t}-3,2^{t}-2)$ & 3  & 2 \\
\hline
\cite{ACH} Theorem 22& $(4p(p+2),2p(p+2)-1,p(p+2)-2,p(p+2)-1)$ & 3  & $4p(p+2)-3$\\
\hline
\cite{DHM} Section III& $(2q,q-1,\frac{q-3}{2},\frac{3(q-1)}{2})$ & $\frac{3}{2}$  & $2q-1$ \\
\hline
\cite{DHM} Section III& $(2q,q,\frac{q-1}{2},\frac{3q-1}{2})$ & 0 & $2q$ \\
\hline
this paper, Remark \ref{re564}              &$(q-1,\frac{q-3}{2},\frac{q-7}{4})$ & 2  & $q-3$ \\
\hline
this paper, Theorem \ref{th999}$^{*}$               &$(q^{2}+q,q+1,1)$ & 1  & $q^{2}+q$ \\
\hline
\multicolumn{4}{l}{{\footnotesize * Apply \ref{th999} to a projective plane of order $q$.}}\\
\end{longtable}

\begin{theorem}\label{th22} Let $A$ be a $v \times v$ incidence matrix of the symmetric incidence structure $(G,\mathcal{B})$ obtained from the development of some $k$-subset $D$ in the Abelian group $G$ (where $\left|G\right|=v$) with difference levels $\mu_{1} <\cdots <\mu_{s}$. Suppose that $k \equiv\mu_{1}\equiv\cdots\equiv\mu_{s} ($mod $2)$. \begin{enumerate}
\item If $k$ is even the binary code of length $v$ with generator matrix $A$ is self-orthogonal.
\item If $k$ is odd the matrix \[ \left[ \begin{array}{cc}
1 & \\
\vdots & A \\
1 & \\
\end{array}\right] \] generates a binary self-orthogonal code of length $v+1$.
\end{enumerate}
\end{theorem}
\begin{proof} By Lemma \ref{le1} we can see that, in both cases, the weights of the rows of the generator matrix are all even and the inner product of any two rows is even as well.
\end{proof}

 We will refer to an incidence structure $(V,\mathcal{B})$ whose incidence matrix generates a self-orthogonal code simply as {\it self-orthogonal}.
 \par
 We will use the following lemma.
 \begin{lemma}\label{le30} Let $(G,\mathcal{B})$ be a symmetric incidence structure coming from the development of a $k$-subset $D$ of the Abelian group $G$ (where $\left|G\right|=v$) with difference levels $\mu_{1}$ and $\mu_{2}$. Let $t$ denote the number of members of $G-\{0\}$ which appear $\mu_{1}$ times in the multiset $\{x-y \mid x,y \in D, x \neq y \}$. The the number of pairs of points in $G$ appearing in exactly $\mu_{1}$ blocks in $\mathcal{B}$ is $\frac{vt}{2}$ and the number of pairs of points of $V$ appearing in $\mu_{2}$ blocks is $\frac{v(v-1-t)}{2}$.
 \end{lemma}
 \begin{proof} For each $x \in V$, there are $t$ points in $V-\{x\}$ each appearing together with $x$ in exactly $\mu_{1}$ blocks. Thus, there are $\frac{vt}{2}$ pairs of points of $V$ appearing in $\mu_{1}$ blocks. Similarly, there are $\frac{v(v-1-t)}{2}$ pairs of points of $V$ appearing in $\mu_{2}$ blocks. It is easily seen that $\frac{vt}{2}+\frac{v(v-1-t)}{2}=\binom{v}{2}$.
 \end{proof}

The following theorem relates these numbers to the dual minimum distance of the binary code generated by the incidence matrix of certain self-orthogonal incidence structures with two distinct $2$-levels.

\begin{theorem} Let $A$ be the incidence matrix of a self-orthogonal incidence structure $(G,\mathcal{B})$ coming from the development of a $k$-subset $D$ of the Abelian group $G$ (where $\left|G\right|=v$) with difference levels $\mu_{1}$ and $\mu_{2}$. Let $t$ denote the number of members of $G-\{0\}$ which appear $\mu_{1}$ times in the multiset $\{x-y \mid x,y \in D, x \neq y \}$. The dual of the binary code with generator matrix $A$ has minimum distance \[ d \geq \frac{(\mu_{2}+k)+\sqrt{(\mu_{2}+k)^{2}+4 \mu_{2}(\mu_{2}-\mu_{1})vt}}{2 \mu_{2}}.\]
\end{theorem}

\begin{proof} Let $S$ be a minimal set of linearly dependent columns of $A$. Then every row of $A$ must intersect an even number of these columns in $1$s. Let $n_{i}$ denote the number of rows of $A$ intersecting exactly $i$ columns of $S$ in $1$s. Let $d=|S|$. Since every column of $A$ contains $k$ $1$s (because the incidence structure $(G,\mathcal{B})$ is symmetric) and the scalar product (over the reals) of any two columns is either $\mu_{1}$ or $\mu_{2}$, using Lemma \ref{le30} we have \[ \sum 2i n_{2i} = kd \] and \[ \sum 2i (2i-1) n_{2i} = \mu_{2} d (d-1) -  (\mu_{2} - \mu_{1}) vt. \] Subtracting the first equation from the second we have \[ \sum 2i (2i-2) n_{2i}=d((d-1)\mu_{2}-k)-(\mu_{2}-\mu_{1})vt\geq0.\] On one hand we get that $d((d-1)\mu_{2}-k)\geq(\mu_{2}-\mu_{1})vt\geq0$ and on the other hand we get that $d^{2}\mu_{2}-d(\mu_{2}+k)-(\mu_{2}-\mu_{1})vt\geq0$. The result follows from solving the quadratic.
\end{proof}

\section{Concluding Remarks}\label{sec6}
We have investigated some generalizations of combinatorial designs arising from almost difference sets, especially the $t$-adesigns. We have discussed some of their basic properties as well as given constructions, some of which correspond to new almost difference sets, and some of which correspond to new almost difference families. Interestingly, we have carried out many searches using cyclotomic cosets of many different orders and the three families constructed in Section \ref{sec3} were the only ones we could find which, after adjoining a point to certain blocks, produce $t$-adesigns.

\bibliographystyle{plain}
\bibliography{myref2}

\end{document}